\documentclass{amsart}

\usepackage{amsfonts, amsmath, amscd}
\usepackage[psamsfonts]{amssymb}
\usepackage{amssymb}

\usepackage{hyperref}
\usepackage{verbatim}

\usepackage{enumerate, amsthm}

\newtheorem{theorem}{Theorem}[section]
\newtheorem{corollary}[theorem]{Corollary}
\newtheorem{proposition}[theorem]{Proposition}

\newtheorem{remark}[theorem]{Remark}
\theoremstyle{definition}

\newtheorem{definition}[theorem]{Definition}

\newcommand{\RR}{\mathbb{R}}
\newcommand{\QQ}{\mathbb{Q}}
\newcommand{\CC}{\mathbb{C}}
\newcommand{\NN}{\mathbb{N}}
\newcommand{\AAA}{\mathbb{A}}
\newcommand{\KK}{\mathbb{K}}

\newcommand{\ZZ}{\mathbb{Z}}
\newcommand{\FF}{\mathbb{F}}
\newcommand{\ee}{\mathrm{e}}
\newcommand{\EE}{\mathbb{E}}

\newcommand{\ssk}{\smallskip}
\newcommand{\msk}{\medskip}

\newcommand{\cc}{{\mathfrak{c}}}   %%
\newcommand{\PP}{\mathbb{P}}
\newcommand{\TQ}{$\mathcal{T}$Q}

 \DeclareMathOperator{\card}{\text{card}}

\title{Topological Transcendental Fields}

\author{Taboka Prince Chalebgwa and Sidney A. Morris}
\address{The Fields Institute for Research in Mathematical Sciences,
222 College Street, Toronto, Ontario, MST 3J1,
Canada}
\email{taboka@aims.ac.za}
\address{School of Engineering, IT and Physical Sciences,
Federation University Australia,
PO Box 663,
Ballarat, Victoria, 3353,
Australia
and\newline
Department of Mathematical and Physical Sciences,
La Trobe University,
Melbourne, Victoria, 3086,
Australia}
\email{morris.sidney@gmail.com}

\thanks{This paper is dedicated to Alf van der Poorten who introduced the second author to transcendental number theory.  The second author also thanks John Loxton for suggesting he look at topological transcendental fields. The first author's research is supported by the Fields Institute for Research in Mathematical Sciences, via the Fields-Ontario postdoctoral Fellowship.}

\begin{document}

\maketitle

\begin{abstract}This article initiates the study of  topological transcendental  fields $\FF$ which are   subfields of the topological  field $\CC$  of  all complex numbers such that $\FF$  consists of only rational numbers and a nonempty set of transcendental numbers.   $\FF$, with the topology it inherits as a subspace of $\CC$, is a topological field.  Each topological  transcendental field is a separable metrizable zero-dimensional space and algebraically is  $\QQ(T)$, the extension of the field of rational numbers by a set $T$ of transcendental numbers.   It is proved that there exist precisely $2^{\aleph_0}$ countably infinite topological transcendental  fields and each is homeomorphic  to the space $\QQ$ of rational numbers with its usual topology.
It is also shown that there is a class of  $2^{2^{\aleph_0} }$ of topological  transcendental  fields  of the form $\QQ(T)$ with $T$ a set of Liouville numbers,  no two of which are homeomorphic. \end{abstract}

%\subjclass[2010]{Primary 46A03,  	54D65}

\section{Preliminaries}

We begin by setting out our notation and making some simple preliminary observations.

\begin{remark}\label{Remark 1}\rm
We shall discuss four fields: $\CC$, the field of all complex numbers; $\RR$, the field of all real numbers; $\AAA$, the field of all algebraic numbers; and $\QQ$, the field of all rational numbers. Observe the following easily verified facts:
\begin{itemize}
\item[\rm(i)] the fields $\CC$ and $\RR$ have cardinality $\cc$, the cardinalty of the continuum; 
\item[\rm(ii)] the fields $\AAA$ and $\QQ$ have cardinality $\aleph_0$; 
\item[\rm(iii)]  $\CC$  with its euclidean topology is homeomorphic   to $\RR\times \RR$, where $\RR$ has its euclidean topology; 
\item[\rm(iv)] each of these four fields has a natural topology; $\CC$  and $\RR$ have euclidean topologies, while $\AAA$ and $\QQ$ inherit a natural topology as a subspace of $\CC$;
\item[\rm(v)] the field $\QQ$ is a dense subfield of the topological field $\RR$ (that is, the closure, in the topological sense, of $\QQ$ is $\RR$);
\item[\rm(vi)]   the topological field $\AAA$ is a dense subfield of the topological field $\CC$; 
\item[\rm(vii)]  $\CC\supset \AAA \supset\AAA\cap\RR \supset \QQ$, but $\AAA$ is not a subset of $\RR$;
\item[\rm(viii)] the field $\CC$ is a vector space of dimension $\cc$  over $\AAA$ and it is also a vector space of dimension $\cc$ over $\QQ$; 
\item[\rm(ix)] $\AAA$ is  a vector space of countably infinite dimension over $\QQ$;
\item[\rm(x)]  $\NN$ denotes the set of positive integers and $\ZZ$ denotes the set of all integers, each with the discrete topology;
\item[\rm(xi)]  $\mathcal{T}$ is  the topological space of all transcendental numbers, where $\mathcal{T}= \CC\setminus \AAA$ and has a natural topology as a subspace of $\CC$.  The topology of $\mathcal{T}$ is separable, metrizable, and zero-dimensional. Also the cardinality of $\mathcal{T}$ is $\cc$ and $\mathcal{T}$ is dense in $\CC$.
\end{itemize}
\end{remark}

\begin{remark}\rm\label{Remark 1A} Now we mention some not so easily verified known results:
\begin{itemize}
\item[\rm(i)]  $\mathcal{T}$ is homeomorphic to  the space $\PP$ of all irrational real numbers. $\PP$ is also homeomorphic to the countably infinite product  $\NN^\NN$. (See  \cite[\S1.9]{vanMill}.)
\item[\rm(ii)]  \TQ\     denotes the set $\mathcal{T}\cup \QQ$.  It is also homeomorphic to $\PP$.
\item[\rm(iii)] Kurt Mahler in 1932 classified the set of all transcendental numbers $\mathcal{T}$ into three disjoint classes: S, T, and U. For a discussion of this important classification, see \cite[Chapter 8]{Baker}. It has been proved that each of these sets has cardinality $\cc$. Further, the Lebesgue measure of T and U are each zero. So S has full measure, that is its complement has measure zero.
\item[\rm(iv)] We introduce the classes SQ = S $\cup$ $\QQ$, TQ = T $\cup$ $\QQ$, UQ = U $\cup$ $\QQ$. Clearly SQ, TQ, and UQ each have cardinality $\cc$, TQ and UQ have measure zero, and SQ has full measure.

\item[\rm(v)] In 1844 Joseph Liouville showed that all members of  a certain class of numbers, now known as the Liouville numbers, are transcendental. A real number $x$  is said to be a\emph{ Liouville number} if for every positive integer $n$, there exists a pair $(p,q)$ of integers with $q>1$, such that $0<|x-\frac{p}{q}|<\frac{1}{q^n}$. (See \cite{Angell}.)  The Liouville numbers are a subset of the Mahler class U. We denote the set of Liouville numbers by L and the set L $\cup \QQ$ by LQ.
\end{itemize}
\end{remark}

Recall the following definitions from \cite{Weintraub}. While Weintraub stated the definitions and propositions using countably infinite sets, there is no problem to state these using sets of any cardinality.

\begin{definition}\label{Definition 1}\rm Let $\EE$ be an extension field of $\FF$. Then $\alpha\in \EE$ is said to be \emph{transcendental over $\FF$} if $\alpha$ is not a root of any polynomial $p(X)\in \FF[X]$, the ring of polynomials over $\FF$ in the variable $X$ with coefficients in $\FF$. The quantity $\alpha\in \EE\setminus \FF$ is said to be \emph{algebraic over $\FF$} if it is not transcendental over $\FF$.
\end{definition}

\begin{definition}\label{Definition 2}\rm  An extension field $\EE$ of a field $\FF$ is said to be a \emph{completely transcendental extension of $\FF$} if $\alpha$ is transcendental over $\FF$,
for every $\alpha\in \EE\setminus \FF$.
\end{definition}

\begin{definition}\label{Definition 3}\rm  Let $\EE$ be an extension field of the field $\FF$. Then $\EE$ is a \emph{purely transcendental extension of $\FF$} if $\EE$ is isomorphic to the field of rational functions $\QQ(\{X_i:i\in I|\})$  of variables $\{X_i:i\in I\}$, where $I$ is a finite or infinite index set.
\end{definition}

\begin{definition}\label{Definition 4}\rm
Let the field $\EE$ be an extension of the field $\FF$. If $I$ is any index set, the subset  $S=\{s_i:i\in I\}$ of $\EE$ is said to be \emph{algebraically independent} over $\FF$ if  for all polynomials $p(X_1,X_2,\dots, X_n)\in F[X_i: i\in I]$, $n\in \NN$,  $p(s_1,s_2,\dots, s_n)\ne 0$, for all $s_1,s_2,\dots, s_n\in S$. By convention if $S=\emptyset$, then $S$ is said to be algebraically independent over $\FF$. \end{definition}

\begin{remark}\rm\label{Remark 1A}
Observe that if a set $S$ is algebraically independent over $\QQ$, then it is algebraically independent  over $\AAA$. Further, algebraic independence implies linear indpendence.\end{remark}

\begin{remark}\rm\label{3A} Central to his definition of  the classes $S$, $T$, and $U$, was the feature that Mahler wanted, namely that any two algebraically dependent transcendental numbers lie in the same class, S, T, or U.  
\end{remark}

We shall use \cite[Lemma 6.1.5]{Weintraub} and \cite[Lemma 6.1.8]{Weintraub} which are stated here as Proposition~\ref{Proposition 2} and Proposition~\ref{Proposition 1}. 

\begin{proposition}\label{Proposition 1} Let $\EE$ be a purely transcendental extension of a field $\FF$. Then $\EE$ is a completely transcendental extension of $\FF$. 
\end{proposition}

\begin{proposition}\label{Proposition 2} Let $\EE$ be an extension field of the field $\FF$ and let $S=\{s_i:i\in I\}$ be  algebraically independent over $\FF$, where $I$ is an index set. Then the extension field $\FF(S)$ is a purely transcendental field.
\end{proposition}

Recall the following definition from, for example, \cite{Warner, Shell}:

\begin{definition}\label{Definition 5}\rm A field $\FF$ with a topology $\tau$ is said to be a \emph{topological field} if the  field operations:
\begin{itemize}
\item[(i)] $(x,y) \to x+y$ from $\FF\times \FF$ to $\FF$,
\item[(ii)] $x\to -x$ from $\FF$ to $\FF$,
\item[(iii)] $(x,y) \to xy$ from $\FF\times \FF$ to $\FF$, and
\item[(iv)] $x\to x^{-1}$ from $\FF$ to $\FF$
are all continuous.
\end{itemize}
\end{definition}

The standard examples of topological fields of  characteristic $0$ are $\RR$, $\CC$, and $\QQ$ with the usual euclidean topologies. Indeed the only connected locally compact  Hausdorff fields  are $\RR$ and $\CC$. However Shakhmatov in \cite{Shakhmatov} proved the following beautiful result: (See also \cite{Shakhmatov1999}.)

\begin{theorem}\label{Theorem 1}
On every field  $\FF$ of infinite cardinality $\aleph$, there exist precisely $2^{2^\aleph}$ distinct topologies which make $\FF$ a topological field.
\end{theorem}

Motivated by the definition of a transcendental group introduced in \cite{Morris2023}, we define here the notion of a {\it topological transcendental field}.

\begin{definition}\label{Definition 6}\rm  The topological field $\FF$ is  said to be a \emph{topological transcendental field} if algebraically it is a subfield of $\CC$, is a subset of $\QQ\cup \mathcal{T}$, and has the topology it inherits as a subspace of $\CC$.
\end{definition}

\begin{remark}\label{Remark 2}\rm Of course the underlying field of a topological transcendental field is a completely transcendental extension of $\QQ$.
\end{remark}

\section{Countably Infinite Transcendental Fields} 

\msk

\begin{proposition}\label{Proposition 2A} If $t$ is any transcendental number, then $\QQ(t)$ is a topological transcendental field.
\end{proposition}

\begin{proof}  This proposition is an immediate consequence of Proposition~\ref{Proposition 2} and Proposition~\ref{Proposition 1}.
\end{proof}

\begin{remark}\label{Remark 5}\rm Of course it is not true that if $t_1$ and $ t_2$ are transcendental, then $\QQ(t_1,t_2)$ is necessarily a transcendental field. For example if $t_1=\pi$ and $t_2=\pi+\sqrt{2}$, then  $\QQ(t_1,t_2)$ is not a transcendental field as $\sqrt{2}\in\QQ(t_1,t_2)$. In fact Paul Erdos \cite{Erdos} proved that for every real number $r$ there exist  Liouville numbers $t_3,t_4, t_5,t_6$ such that $t_3\cdot t_4=r$ and $t_5+t_6=r$.  Indeed he proved that for each real number $r$, there are uncountably many Liouville numbers   $t_3, t_4$ and $t_5,t_6$  with these properties. 
\end{remark}

Having established the existence of countably infinite topological transcendental fields, we now describe a very concrete example.
But first we state a well-known theorem on transcendental numbers, see Theorem 1.4 and the comments following it,  in \cite{Baker}.

\begin{theorem}\label{Theorem 2}\ {\rm[Lindemann-Weierstrass  Theorem]}\quad For $m\in \NN$,  any algebraic numbers $\alpha_1,\alpha_2,\dots,\alpha_m$ which are linearly independent over $\QQ$, the numbers $\ee^{\alpha_1},\ee^{\alpha_2},\dots,\ee^{\alpha_m}$ are algebraically independent.
\end{theorem}

\begin{theorem}\label{Theorem 3}
Let $S=\{\alpha_1,\alpha_2,\dots,\alpha_n,\dots\}$ be a countably infinite set of algebraic numbers which are linearly independent over $\QQ$.   If  $T=\{\ee^{\alpha_1}, \ee^{\alpha_2},\dots,\ee^{\alpha_n},\dots\}$,  then $\QQ(T)$,  is a topological transcendental field.
\end{theorem}

\begin{proof} By Proposition~\ref{Proposition 2} and Proposition~\ref{Proposition 1},  $\QQ(T)$ is a topological transcendental field.
\end{proof}

\begin{theorem} \label{Theorem 4}There exist precisely $2^{\aleph_0}$ countably infinite topological  transcendental fields, each of which is homeomorphic to $\QQ$.
\end{theorem}

\begin{proof} Using the notation of  Theorem~\ref{Theorem 3}, there are $2^{\aleph_0}$ subsets of $T$ and, due to algebraic independence, any two such subsets $V, W$, $V\ne W$, are such that $\QQ(V)\ne \QQ(W)$. 

Further, there are only $2^{\aleph_0}$ countably infinite subsets of $\CC$.  So there exist precisely $2^{\aleph_0}$ countably infinite topological  transcendental groups.

By \cite[Theorem 1.9.6]{vanMill} the  space $\QQ$ of all rational numbers
 up to homeomorphism is the unique nonempty countably infinite separable metrizable space without isolated points. In a topological field (indeed in a topological group) there are isolated points if and only if the topological field has the discrete topology.  But by \cite[Theorem 6]{Morris1977} the only discrete subgroups of $\CC$ are isomorphic to $\ZZ$  and $\ZZ\times \ZZ$, neither of which has the algebraic structure of a field.  So every countably infinite topological transcendental field is homeomorphic to $\QQ$.
\end{proof}

\section{Topological Transcendental Fields of Continuum Cardinality} 

\msk 

\begin{theorem}\label{4}
Let $\KK$  be a topological transcendental field  of cardinality $\card(\KK)$.   Then the extension field $\KK(t)$ is a topological transcendental field for all but a set of cardinality $\card(\KK)$ of $t\in \CC$.
\end{theorem}

\begin{proof} 
There are a countably infinite number of  members of $\KK(t)$  of the form $$z=\frac{c_0+c_1t+c_2t^2+\dots c_nt^n}{d_0+d_1t+d_2t^2+\dots d_mt^m},$$ 
for $c_0,c_1,\dots,c_n,d_0,d_1,\dots,d_m\in \KK$, $n,m\in\NN$. If $z$ is an algebraic number $a$, then
$$c_0+c_1t+c_2t_2+\dots c_nt^n- ad_0-ad_1t-ad_2t^2-\dots ad_mt^m=0.$$
For given $ c_0,c_1,\dots,c_n,d_0,d_1,\dots d_m$, and each of the countably infinite values of $a$, the Fundamental Theorem of Algebra says that there at most $\max(n,m)$  transcendental number solutions  for $t$. So $\KK(t)$ consists of only rational numbers and transcendental numbers except for at most $\aleph_0\times \card(\KK)=\card(\KK)$ values of $t$, which proves the theorem.  \end{proof}

Noting our Remark~\ref{Remark 5}, Corollary~\ref{5A} is of interest.

\begin{corollary}\label{5A}  If $t_1,t_2$ are transcendental numbers, then $\QQ(t_1,t_2)$ is a topological transcendental field for all but a countable infinite number of pairs $(t_1,t_2)$.  Indeed if $W$ is a countable set of transcendental numbers, then $\QQ(W)$ is a topological transcendental field for but a countably infinite number of sets $W$.  \qed\end{corollary}

\begin{corollary}\label{6}  Let $\KK$  be a topological transcendental field of cardinality $\aleph<2^{\aleph_0}$. Then there exists a $t\in\CC$ such that $\KK(t)$ is a topological  transcendental field which properly contains $\KK$.\qed
\end{corollary}

\begin{theorem}\label{7} Let $E$ be any set of cardinality $\cc$ of transcendental numbers. Then there exists a topological transcendental field $\QQ(T)$ of cardinality $\cc$, where $T\subseteq E$. Further $\QQ(T)$ has $2^\cc$ distinct topological  transcendental subfields. \end{theorem}

\begin{proof}
Consider the set $\mathcal F$ of all  topological transcendental fields $\QQ(F)$, where $F$ is a subset of $E$,  with the property  that for each pair   $W,V\subset F$ such that  $ W\ne V$, $\QQ(V)\ne \QQ(W)$. 

By Corollary~\ref{5A} and the fact that  $E$ is uncountable, there exist  $s,t\in E$,  $t\notin \QQ(s)$, $s\notin \QQ(t)$, and  $\QQ(s,t)$ is  a  topological transcendental field. Then $\QQ(s,t)\in \mathcal F$.

Put a partial order on the members of $\mathcal F$ by set theory containment. Consider any totally ordered subset $\mathcal S$ of  members of $\mathcal F$. Let  $\KK$ be the union of members of $\mathcal S$. Clearly it is a member of $\mathcal F$ and is an upper bound of $\mathcal S$. Therefore by Zorn's Lemma, $\mathcal F$ has a maximal member $\QQ(T)$, where $T\subseteq E$.

Suppose $T$ has cardinality $\aleph<\cc$, then by the proof of Theorem~\ref{4}, there exists an $e\in E$, such that $\QQ(T)(e)=\QQ(T,\{e\}) $ is a topological transcendental field
which is easily seen to be a member of $\mathcal F$.  This contradicts the maximality of $\QQ(T)$. So $T$ has cardinality $\cc$. 

Further by the definition of $\mathcal F$,  $\QQ(T)$ has $2^\cc$ distinct topological  transcendental subfields.
\end{proof}

\begin{theorem}\label{8}
Let  $E$ be a set of transcendental numbers of cardinality $\cc$.  Then there exist  $2^\cc$ topological  transcendental fields $\QQ(T)$, where $T\subseteq E$, and no two of the topological  transcendental fields are homeomorphic.
\end{theorem}

\begin{proof} By the Laverentieff Theorem,  Theorem A8.5 of \cite{vanMill}, there are at most $\cc$ subspaces of $\CC$ which are homeomorphic. 
So from Theorem~\ref{7} there are $2^\cc$ topological transcendental fields no two of which are homeomorphic. 
\end{proof}

\begin{corollary}\label{9} Let $E$ be the set $L$ of Liouville numbers or the Mahler set $U$ or the Mahler set $T$ or the Mahler set $S$. Then there exist  $2^\cc$ topological  transcendental fields $\QQ(T)$, where $T\subseteq E$, and no two of the topological  transcendental fields are homeomorphic. \qed
\end{corollary}

 \noindent\bf Open Question 1.\rm\quad Does there exist a topological transcendental field of full measure.? (A set is said to be of full measure if its complement has measure zero.)
 
\ssk



\begin{thebibliography}{999}


\bibitem{vanMill} J. van Mill, {\it The Infinite-Dimensional Topology of Function Spaces}. Elsevier, 2001.


\bibitem{Baker} A. Baker, {\it Transcendental Number Theory}. Cambridge University Press, 1975.

\bibitem{Angell} D. Angell, {\it Irrationality and Transcendence in Number Theory,} CRC Press, 2022. 



\bibitem{Weintraub} S. H.  Weintraub, {\it Galois Theory, second edition}, Springer, 2009.



\bibitem{Warner} S. Warner, {\it Topological Fields}. North Holland, 1989.

\bibitem{Shell} N. Shell, {\it Topological Fields and Near Valuations}, Routledge, 1990. 

\bibitem{Shakhmatov} D.B. Shakhmatov, Cardinal invariants of topological fields, Dokl. Akad. Nauk. SSSR {\bf 1983}, {\em 271}, 1332--1336.

\bibitem{Shakhmatov1999} D.B. Shakhmatov, A comparative study of selected results and open problems concerning topological groups, fields and vector spaces, {\em Topology and its Applic.} {\bf 1999}, {\em 91}, 51--63.



\bibitem{Morris2023} S.A. Morris, Transcendental Groups, to appear.  \url{https://arxiv.org/abs/2112.12450}

\bibitem{Erdos} P. Erdos, Representations of real numbers as sums and products of Liouville numbers, {\em Michigan Math. J.}{\bf 1962}, {\em 9}, 59--60.


\bibitem{Morris1977} S. A. Morris, {\it Pontryagin Duality and the Structure of Locally Compact Abelian  Groups}. Cambridge University Press, 1977.


\end{thebibliography}
\end{document}